\documentclass[1p]{elsarticle}
\usepackage{amsmath,stmaryrd,color,amsfonts,amsthm}
\usepackage{cases,mathrsfs,amssymb,setspace}
\usepackage{lineno,hyperref}
\modulolinenumbers[5]

\journal{ arXiv 
}
\newtheorem{thm}{Theorem}[section]
\newtheorem{lem}[thm]{Lemma}
\newtheorem{example}[thm]{Example}

\newtheorem{definition}[thm]{Definition}
\newtheorem{proposition}[thm]{Proposition}

\newtheorem{remark}[thm]{Remark}

\begin{document}

\begin{frontmatter}

\title{Strong limit theorems for weighted sums of negatively associated random variables in nonlinear probability
\tnoteref{mytitlenote}}
\tnotetext[mytitlenote]{Yuting Lan's work was supported by 
$\times\times\times$.
 We gratefully acknowledge the  discussions with Prof. Zengjing Chen.
}

\author[mymainaddress]{Yuting Lan\corref{mycorrespondingauthor}}
\cortext[mycorrespondingauthor]{Corresponding author}
\ead{lan.yuting@mail.shufe.edu.cn}

\author[mysecondaryaddress]{Ning Zhang}
\ead{nzhang@mail.sdu.edu.cn}

\address[mymainaddress]{School of Statistics and Management, Shanghai University of Finance and Economics,\\
Shanghai 200433, China}

\address[mysecondaryaddress]{School of Mathematics, Shandong University, Jinan 250100, China}

\begin{abstract}
In this paper, based on the initiation of the notion of negatively associated random variables under nonlinear probability, a strong limit theorem for weighted sums of random variables within the same frame is achieved without assumptions of independence and identical distribution, from which the Marcinkiewich-Zygmund type and Kolmogorov type strong laws of large numbers are derived. In addition, as applications of our results, Stranssen type invariance principles of negatively associated random variables and vertically independent random variables are proposed respectively.
\end{abstract}

\begin{keyword}
nonlinear probability\sep negative association\sep Marcinkiewich-Zygmund  strong law of large numbers\sep Kolmogorov strong law of large numbers\sep  Stranssen type invariance principles
\MSC[2010] 60F15
\end{keyword}

\end{frontmatter}

\linenumbers

\section{Introduction}

It is well-known that strong laws of large numbers (SLLNs) as fundamental limit theorems in probability theory play an important role in the development of probability theory and its application. The additivity of probability and expectation is the basis of the proofs of these classic SLLNs.
However, due to the fact that many uncertain phenomena cannot be well modeled and interpreted by additive probability and expectation, such assumption of additivity drops its reasonability in certain applications from various fields, such as mathematical economics, statistics, quantum mechanics and finance, etc. Therefore, numerous papers (\cite{CZEL02,HP73,PS99,SD89,WP01,WPFT82,WLKJ90}) recently describe and interpret those phenomena using non-additive/imprecise probabilities or nonlinear expectations, such as  Choquet expectation \cite{CG54}, g-expectation \cite{CZCT05,CFHY02}, G-expectation \cite{PS07a,PS10}.

In classic probability, Kolmogorov  SLLN(1930) states that $\lim_{n\to \infty}\frac{\sum_{i=1}^{n}X_i}{n}=E[X_1]$ a.s. for independent and identically distributed  random variables.
Under  nonlinear probability framework, scholars redefine several different concepts of independence and identical distribution of random variables and then obtain general results of Kolmogorov type SLLNs under non-additivity probabilities. Some typical results can be described as following:

Let $\mathcal{P}$ be a family of probability measures on measurable space $(\Omega,\mathcal{F})$ and $(\mathbb{V},v)$be a pair of upper-lower probabilities derived by $\mathcal{P}$:
\[\mathbb{V}(A)=\sup_{P\in\mathcal{P}}P(A),\quad v(A)=\inf_{P\in\mathcal{P}}(A),\quad A\in \mathcal{F}.\]
The corresponding Choquet expectations $(C_{\mathbb{V}},C_v)$ are defined by
$$C_{\mathbb{V}}[X]:=\int_{0}^{\infty}\mathbb{V}(X\geq t)+\int_{-\infty}^{0}[\mathbb{V}(X\geq t)-1],\quad C_{v}[X]:=\int_{0}^{\infty}v(X\geq t)+\int_{-\infty}^{0}[v(X\geq t)-1].$$
The pair of conjugative upper-lower expectations derived by $\mathcal{P}$, denoted as $(\mathbb{E},\mathcal{E})$, can be defined by:
$$\mathbb{E}[X]=\sup_{P\in\mathcal{P}}E_{P}[X],\quad
\mathcal{E}[X]=\inf_{P\in\mathcal{P}}E_{P}[X].$$
Here and in the sequel, $E_P$ denotes the classical expectation corresponding to the classical probability $P$.

Within the nonlinear framework, applying various assumptions, many scholars achieve following SLLN, that is, under lower probability $v$, any cluster point of empirical averages lies between the lower Choquet expectation $C_v$ and the upper Choquet expectation $C_{\mathbb{V}}$  with probability one:
$$v\bigg( \omega\in\Omega: C_{v}[X_1]\leq\liminf_{n\to\infty}\frac{1}{n}\sum_{i=1}^{n}X_i(\omega)
\leq\limsup_{n\to\infty}\frac{1}{n}\sum_{i=1}^{n}X_i(\omega)\leq C_{\mathbb{V}}[X_1]\bigg)=1.\eqno{(1.1)}$$
F. Maccheroni and M. Marinacci \cite{MFMM05}, M. Marinnacci \cite{MM99} achieve their results by assuming that the random variables are bounded and continuous and  $v$ is a totally monotone capacity on Polish space $\Omega$. L. Epstein and D. Schneider \cite{ELSD03} obtain the same results under the assumption that $\mathbb{V}$ is rectangular. G. Cooman and E. Miranda \cite{CGME08} initiate a similar results requiring that the random variables are uniformly bounded and F. Cozman \cite{CF10} generalizes this results requiring the bounded variance of random variables. P. Ter$\acute{a}$n's \cite{TP14} SLLN is established with the assumptions that $v$ is a completely monotone capacity and also is a continuous capacity or topological capacity.

Generally, for any random variable $X$, the Choquet expectations and upper-lower expections have the following relationship:
$$C_{v}[X]\leq\mathcal{E}[X]\leq \mathbb{E}[X] \leq C_{\mathbb{V}}[X],\eqno{(1.2)}$$
which indicates that the gap between Choquet expectations $C_{\mathbb{V}}[X_1]$ and $C_v[X_1]$ is larger than  that between upper-lower expectations $\mathbb{E}[X_1]$ and $\mathcal{E}[X_1]$.
Thus Z. Chen et al. \cite{CZWP13} recently achieve a more precise SLLN based on upper-lower expectations:
\[v\left( \omega\in\Omega:\mathcal{E}[X_{1}]\leq \liminf_{n\rightarrow\infty}\frac{1}{n}\sum_{i=1}^{n}X_{i}(\omega)   \leq \limsup_{n\rightarrow\infty}\frac{1}{n}\sum_{i=1}^{n}X_{i}(\omega) \leq\mathbb{E}[X_{1}] \right)=1,\eqno{(1.3)}\]
with the assumptions of vertically independent random variables with $\mathbb{E}[X_{n}]=\mathbb{E}[X_{1}]$ and $\mathcal{E}[X_{n}]=\mathcal{E}[X_{1}]$ for any $n\in\mathbb{N^{+}}$.

Motivated by above works, we aim to generalize the Kolmogorov SLLN under nonlinear probabilities. In classic probability, the Kolmogorov SLLN has been extended in three main directions. Firstly, leaving out the assumption of identical distribution, Kolmogorov achieve another type of SLLN, that is $\lim_{n\to \infty}\frac{\sum_{i=1}^{n}(X_i-E[X_i])}{n}=0$ \ a.s. with $\sum_{n=1}^{\infty}\frac{Var(X_i)}{n^2}<\infty$.
In addition, the Marcinkiewicz-Zygmund type SLLNs are obtained, stating that  $\lim_{n\to \infty}\frac{\sum_{i=1}^{n}(X_i-E[X_i])}{n^{1/p}}=0$ \ a.s. under some conditions with $1<p<2$ (see \cite{CYLT73,CYRH61,GA78} etc.). The Marcinkiewicz-Zygmund type SLLNs describe the convergent speed of difference between average of samples and average of expectations more precisely.
What's more, scholars pay attention to SLLNs for dependent random variables (\cite{BHST82,SQ02}) since  the assumption of independence is not always satisfied nor easily to be verified. T. Chandra and S. Goswami\cite{CTGS96},  N. Etimadi and A. Lenzhen \cite{ENLA03}, S. Sung\cite{SS14} et.al. achieve several distinct Marcinkiewich-Zygmund type SLLNs for negatively dependent, negatively associated or pairwise independent random variables under different moments requirements.

Inspired by the generalization in the classic probability, we focus on the validity of Marcinkiewich-Zygmund type SLLN under nonlinear probability. In this paper, we initiate the notion of negatively associated random variables under nonlinear probability and  achieve the following strong limit theorem for weighted sums of negatively associated random variables:
$$v\left(\liminf_{n\rightarrow \infty}
\frac{\sum_{i=1}^{n}a_{i}(X_{i}-\mathcal{E}[X_{i}])}{A_{n}}\geq0\right)=1,
\quad\quad
v\left(\limsup_{n\rightarrow \infty}
\frac{\sum_{i=1}^{n}a_{i}(X_{i}-\mathbb{E}[X_{i}])}{A_{n}}\leq0\right)=1.
$$
It can be considered as a natural extension of the  Kolmogorov type SLLN under nonlinear probability. Compared with other literatures, our results have five main improvements:
\begin{itemize}
\item  Firstly, since  our theorem is achieved under upper-lower expectations, it is a more precise results than using Choquet expectations. Because the same theorem for Choquet expectations can be naturally deduced from our results by Equation (1.2).
\item Secondly, we only need the property of continuity from above for lower probability $v$, which is a more natural and weaker condition than some earlier researches.
\item Thirdly, we only require the sequence of random variables to be negatively associated. The notion of negative association we initiated is a weaker condition than some existing concepts, such as Peng-independence, vertical independence and forward factorization under nonlinear probabilities. More details can be seen in Section 3.
\item What's more, we do not make any assumptions about the distributions of random variables such as identically distributed or a weaker requirement that $\mathbb{E}[X_{n}]=\mathbb{E}[X_{1}]$ and $\mathcal{E}[X_{n}]=\mathcal{E}[X_{1}]$ for any $n\in\mathbb{N^{+}}$. We only require  that $\sup_{n\geq 1}\mathbb{E}[|X_{n}|^{\alpha+1}]<\infty$ for some constant $\alpha>0$.
\item Finally, the assumption that $\lim_{n\rightarrow \infty}A_{n}/n^{\frac{1}{\beta+1}}=\infty$
for some constant $\beta\in\left(0,\min(1,\alpha)\right)$ enables $n$ in SLLN to be extended to $A_{n}$ which has lower order. In this sense, our strong limit theorem for weighted sums specifies the convergent speed of difference between average of samples and average of expectations under the nonlinear probability.
\end{itemize}
Above five improvements make our results a more natural and fairly neat extension of Kolmogorov SLLN under nonlinear probability. It is a general form of both Kolmogorov and Marcinkiewich-Zygmund type SLLNs under nonlinear probabilities, including those results in the form of Equation (1.1) and Equation (1.3).

The paper is organized as following:  In Section 2, we introduce some preliminaries about nonlinear probabilities. Then we initiate the notion of negative association and illustrate its properties in Section 3. In Section 4, we state and prove the strong limit theorem for weighted sums of negatively associated random variables, from which the Kolmogorov type and Marcinkiewich-Zygmund type SLLNs can be deduced. Finally, in Section 5, we obtain  Strassen type invariance principles of negatively associated random variables  as applications of our results.

\section{Preliminaries }
In this section,  some basic definitions and propositions of nonlinear probability are introduced.

Let $(\Omega,\mathcal{F})$ be a measurable space and $\mathcal{P}$ be a nonempty set of  probability measures on $(\Omega,\mathcal{F})$. For any $A\in\mathcal{F}$, denote
$$\mathbb{V}(A):=\sup_{P\in\mathcal{P}}P(A),\quad\quad
v(A):=\inf_{P\in\mathcal{P}}P(A).$$
$(\mathbb{V},v)$ is called upper probability and lower probability respectively. It is obvious that $\mathbb{V}(A)\geq v(A)$ for any $A\in\mathcal{F}$. Furthermore, $(\mathbb{V},v)$ also satisfy the following properties.
\begin{proposition}(see \cite{CZWP13})
(1) $\mathbb{V}(\emptyset)=v(\emptyset)=0$, $\mathbb{V}(\Omega)=v(\Omega)=1$;

(2)  Monotonicity: $\forall A, B\in\mathcal{F}$, if $A\subset B$, then $\mathbb{V}(A)\leq \mathbb{V}(B)$, $v(A)\leq v(B)$;

(3) Conjugacy: $\mathbb{V}(A)+v(A^{c})=1$, $\forall A\in\mathcal{F}$;

(4) Continuity from below: For upper probability $\mathbb{V}$, if $A_{n}, A\in\mathcal{F}$ and $A_{n}\uparrow A$, then $\mathbb{V}(A_{n})\uparrow \mathbb{V}(A)$;

(5) Continuity from above: For lower probability $v$, if $A_{n}, A\in\mathcal{F}$ and $A_{n}\downarrow A$, then $v(A_{n})\downarrow v(A)$.
\end{proposition}

\begin{remark}Property (4) is equivalent to Property (5). A capacity is called continuous if it both continuous from below and from above. Generally, $\mathbb{V}$ does not satisfy continuity from above.
\end{remark}

Denote
$$\mathcal{L}=\{X|X \ is\ a \ \mathcal{F}-measurable \ random \ variable \ such  \ that\  E_{P}[|X|]<\infty , \
\forall P\in\mathcal{P}\}.$$
Now define the upper expectation  $\mathbb{E}$ and lower expectation $\mathcal{E}$ derived by $\mathcal{P}$.
For any $X\in\mathcal{L}$, define
$$\mathbb{E}[X]:=\sup_{P\in\mathcal{P}}E_{P}[X],\quad\quad
\mathcal{E}[X]:=\inf_{P\in\mathcal{P}}E_{P}[X].$$
$(\Omega,\mathcal{F},\mathcal{P},\mathbb{V})$ is called the upper probability space.
Obvious, for any random variable $X$ on $(\Omega,\mathcal{F},\mathcal{P},\mathbb{V})$, $\mathbb{E}[X]\geq\mathcal{E}[X]$. In addition, $\mathcal{E}[X]=\inf_{P\in\mathcal{P}}E_{P}[X]=
-\sup_{P\in\mathcal{P}}E_{P}[-X]
=-\mathbb{E}[-X]$. Therefore, $\mathbb{E}$ and $\mathcal{E}$ is a pair of conjugate
expectations satisfying the following properties.

\begin{proposition}
For all $X,Y\in\mathcal{L}$, the following properties hold:

(1) Monotonicity: $X\geq Y$ implies $\mathbb{E}[X]\geq\mathbb{E}[Y]$;

(2) Constant preserving: $\mathbb{E}[c]=c$, $\forall c\in\mathbb{R}$;

(3) Sub-additivity: $\mathbb{E}[X+Y]\leq\mathbb{E}[X]+\mathbb{E}[Y]$;

(4) Positive homogeneity: $\mathbb{E}[\lambda X]=\lambda\mathbb{E}[X]$, $\forall\lambda\geq0$.
\end{proposition}

Therefore, the upper expectation $\mathbb{E}$ is also the so-called sublinear expectation initiated by Peng in \cite{PS10}. In addition, the following properties hold.

\begin{proposition} (see \cite{PS10}) For any random variable $X,Y\in\mathcal{L}$, \label{proposition 2.4}

(1) $\mathbb{E}[X]\geq\mathcal{E}[X]$;

(2) $\mathbb{E}[aX]=a^{+}\mathbb{E}[X]+a^{-}\mathbb{E}[-X]$, $\forall a\in\mathbb{R}$;

(3) $\mathbb{E}[X]-\mathbb{E}[Y]\leq \mathbb{E}[X-Y]$;

(4) $\mathbb{E}[X+c]=\mathbb{E}[X]+c$, $\forall c\in\mathbb{R}$.

\end{proposition}

\begin{definition}
(Quasi-surely) A set $D$ is a polar set if $\mathbb{V}(D)=0$ and a property holds ¡°quasi-surely¡±(q.s. for short) if it holds outside a polar set.
\end{definition}
In next, we will list some inequalities on the upper probability space, which can be regarded as an extension of inequalities of classic probability theory.
\begin{proposition} (see\cite{CZWP13}) \label{ineq}Supposing $X,Y\in\mathcal{L}$, then following inequality hold

(1) $H\ddot{o}lder$'s inequality:
For  $p,q>1$ with $\frac{1}{p}+\frac{1}{q}=1$,
$$\mathbb{E}[|XY|]\leq (\mathbb{E}[|X|^{p}])^{\frac{1}{p}}  (\mathbb{E}[|Y|^{q}])^{\frac{1}{q}}.$$

(2) Chebyshev's inequalities: Let $f(x)>0$ be a nondecreasing function on $\mathbb{R}$. Then for any $x$,
$$\mathbb{V}(X\geq x)\leq \frac{\mathbb{E}[f(X)]}{f(x)},\quad\quad v(X\geq x)\leq \frac{\mathcal{E}[f(X)]}{f(x)}.$$

\indent \ \ \ \ Let $f(x)>0$ be a even function on $\mathbb{R}$ and nondecreasing on $(0,\infty)$. Then for any $x>0$,
$$\mathbb{V}(|X|\geq x)\leq \frac{\mathbb{E}[f(X)]}{f(x)},\quad\quad v(|X|\geq x)\leq \frac{\mathcal{E}[f(X)]}{f(x)}.$$

(3) Jensen's inequality: Let $f(\cdot)$ be a convex function on $\mathbb{R}$. Suppose that $\mathbb{E}[X]$ and $\mathbb{E}[f(x)]$ exist. Then,
$$\mathbb{E}[f(x)]\geq f(\mathbb{E}[X]).$$

\end{proposition}

\begin{lem}\label{BClemma}
(Borel-Cantelli Lemma) (see \cite{CZWP13}) Let $\{A_{n}\}_{n=1}^{\infty}$ be a sequence of events in $\mathcal{F}$ and $(\mathbb{V},v)$ be a pair of upper and lower probabilities
generated by $\mathcal{P}$. If $\sum_{n=1}^{\infty}\mathbb{V}(A_{n})<\infty$, then $\mathbb{V}(\bigcap_{n=1}^{\infty}\bigcup_{i=n}^{\infty}A_{i})=0$.
\end{lem}

\section{Negative association and related properties}
 In this section, we initiate the notion of negatively associated random variables on nonlinear probability spaces and illustrate the concept in detail.

Let $C_{l,Lip}(\mathbb{R})$ denote the set of all locally Lipschitz continuous functions, that is the set of all the functions $\varphi(\cdot)$ satisfying
$$|\varphi(x)-\varphi(y)|\leq C (1+|x|^{m}+|y|^{m})|x-y|,\quad \forall x,y\in
\ \mathbb{R}, \ some\ C>0, \ m\in N.$$
On the upper probability space $(\Omega,\mathcal{F},\mathcal{P},\mathbb{V})$, Peng \cite{PS10} initiates the notion of independent random variables under $\mathbb{E}$.
\begin{definition}(\textbf{Peng-independence}) Let $X$ and $Y$ be two random variables on the upper probability space $(\Omega,\mathcal{F},\mathcal{P},\mathbb{V})$. $Y$ is said to be independent from $X$, if for each $\varphi\in C_{l,Lip}(\mathbb{R})$, there is
$$\mathbb{E}[\varphi(X,Y)]=\mathbb{E}[\mathbb{E}[\varphi(x,Y)]|_{x\in X}].$$
$\{X_i\}_{i=1}^{\infty}$ is said to be a sequence of independent random variables
if $X_{n+1}$ is independent from $(X_1,\cdots,X_n)$ for any $n\in\mathbb{N}^{+}$.
\end{definition}

However, the condition of independence of random variables dose not always hold nor it is easily verified in practical applications under both linear and nonlinear probabilities. For more general cases, scholars initiate the definitions of negatively dependence and association under classic probabilities. Motivated by the results under classic probabilities, we initiate the definition of negatively associated random variables on nonlinear probability spaces and derive corresponding limit theorems based on this concept in next section.

 Let $C^+_b(\mathbb{R})$ denote the set of all the nonnegatively bounded and continuous functions.

\begin{definition}(\textbf{Negative association})\label{NA}
Let $X,Y$ be two random variables on the upper probability space $(\Omega,\mathcal{F},\mathcal{P},\mathbb{V})$. $X$ and $Y$ are called negatively associated if for any $f_1(\cdot), f_2(\cdot)\in C^+_{b}(\mathbb{R})$ with the same monotonicity, there is
$$\mathbb{E}[f_1(X)f_2(Y)]\leq \mathbb{E}[f_1(X)]\mathbb{E}[f_2(Y)].$$
$\{X_{i}\}_{i=1}^{\infty}$ is said to be a sequence of negatively associated random variables on the upper probability space $(\Omega,\mathcal{F},\mathcal{P},\mathbb{V})$ if for any $n\geq 1$ and  any  $\{f_i(\cdot)\}_{i=1}^{\infty}\subset C^+_{b}(\mathbb{R})$ with the same monotonicity, there is
$$\mathbb{E}[\prod_{i=1}^{n+1}f_i(X_i)]\leq\mathbb{E}[\prod_{i=1}^{n}f_i(X_i)]\mathbb{E}[f_{n+1}(X_{n+1})].$$

\end{definition}
 Now we will give an example to illustrate that our definition is natural and reasonable.

\begin{example}
Suppose $\{X_{i}\}_{i=1}^{\infty}$ is a sequence of negatively associated random variables under each $P\in\mathcal{P}$, that is
$$Cov \bigg( f(X_{i}:i\in A), \ g(X_{j}:j\in B)\bigg)\leq 0, \  \  \  \  \ \ \ \ \forall P\in\mathcal{P},$$
where $f(\cdot)$ and $g(\cdot)$ are coordinatewise nondecreasing or coordinatewise nonincreasing functions, and $A$ and $B$ are disjointed index sets.
Then $\{X_{i}\}_{i=1}^{\infty}$ is a sequence of negatively associated random variables on the upper probability
space $(\Omega,\mathcal{F},\mathcal{P},\mathbb{V})$.
\end{example}

\begin{proof}
Since  $\{X_{i}\}_{i=1}^{\infty}$ is a sequence of negatively associated random variables under each $P\in\mathcal{P}$, we have
$$E_{P}[f(X_{i}:i\in A)g(X_{j}:j\in B)]
\leq E_{P}[f(X_{i}:i\in A)]\cdot E_P[g(X_{j}:j\in B)],     \ \ \ \ \forall P\in\mathcal{P}.$$
For  all  functions $\{f_i(\cdot)\}_{i=1}^{n+1} \subset C^+_{b}(\mathbb{R})$ with the same monotonicity, define
$$F(x_1,x_2,\cdots,x_n)\overset{\triangle}{=}\prod_{i=1}^{n}f_i(x_i) \ \ and
\ \ G(x)\overset{\triangle}{=}f_{n+1}(x).$$
Then $F(\cdot)$ and $G(\cdot)$ are nonnegatively continuous coordinatewise nondecreasing or coordinatewise nonincreasing functions. Therefore,
\begin{align*}
\mathbb{E}[\prod_{i=1}^{n+1}f_i(X_i)]\displaybreak
&=\sup_{P\in\mathcal{P}}E_{P}[\prod_{i=1}^{n+1}f_i(X_i)]\\
&=\sup_{P\in\mathcal{P}}E_{P}[F(X_1,\cdots,X_n)\cdot G(X_{n+1})]\\
&\leq \sup_{P\in\mathcal{P}}\bigg(E_{P}[F(X_1,\cdots,X_n)] E_{P}[G(X_{n+1})]\bigg)\\
&\leq \sup_{P\in\mathcal{P}}E_{P}[F(X_1,\cdots,X_n)]\cdot \sup_{P\in\mathcal{P}}E_{P}[G(X_{n+1})]\\
&=\mathbb{E}[F(X_1,\cdots,X_n)]\mathbb{E}[G(X_{n+1})]\\
&=\mathbb{E}[\prod_{i=1}^{n}f_i(X_i)]\mathbb{E}[f_{n+1}(X_{n+1})].
\end{align*}
The last inequality follows from that $\{f_i(\cdot)\}_{i=1}^{n+1}$  are all nonnegative functions. Thus, $\{X_{i}\}_{i=1}^{\infty}$ is a sequence of negatively associated random variables on the upper probability
space $(\Omega,\mathcal{F},\mathcal{P},\mathbb{V})$.
\end{proof}

The following proposition shows the relation between Peng-independence and our negative association.
\begin{proposition}
If $\{X_{i}\}_{i=1}^{\infty}$ is a sequence of Peng-independent  random variables on the upper probability space $(\Omega,\mathcal{F},\mathcal{P},\mathbb{V})$, then it is a sequence of negatively associated random variables.
\end{proposition}
\begin{proof}
For any sequence of functions $\{f_i(\cdot)\}_{i=1}^{\infty}\subset C^+_b(\mathbb{R})$ with the same monotonicity, they are naturally locally  Lipschitz continuous  functions. Since $\{X_{i}\}_{i=1}^{\infty}$ is a sequence of Peng-independent random variables and $\{f_i(\cdot)\}_{i=1}^{\infty}$ is a sequence of nonnegative functions, there is
\begin{align*}
&\quad \ \mathbb{E}\big[\prod_{i=1}^{n+1}f_i(X_i)\big]\\
&=\mathbb{E}[\mathbb{E}[ \prod_{i=1}^{n}f_i(x_i) \cdot f_{n+1}(X_{n+1})]
\big|_{x_{1}=X_{1},\cdots,x_{n}=X_{n}}]\\
&=\mathbb{E}[\big(\prod_{i=1}^{n}f_i(x_i) \cdot\mathbb{E}[  f_{n+1}(X_{n+1})]
\big)\big|_{x_{1}=X_{1},\cdots,x_{n}=X_{n}}]\\
&=\mathbb{E}[\prod_{i=1}^{n}f_i(x_i)
\big|_{x_{1}=X_{1},\cdots,x_{n}=X_{n}}]\cdot\mathbb{E}[  f_{n+1}(X_{n+1})]\\
&=\mathbb{E}\big[\prod_{i=1}^{n}f_i(X_i)\big]\cdot\mathbb{E}[  f_{n+1}(X_{n+1})]
\end{align*}
It is obvious that $ \mathbb{E}\big[\prod_{i=1}^{n+1}f_i(X_i)\big]\leq \mathbb{E}\big[\prod_{i=1}^{n}f_i(X_i)\big]\cdot\mathbb{E}[  f_{n+1}(X_{n+1})]$. Therefore,
$\{X_{i}\}_{i=1}^{\infty}$ is a sequence of negatively associated random variables. 
\end{proof}

De Cooman and Miranda \cite{CGME08} introduce the  concept of forward factorization for random variables,  which is  further
discussed in \cite{MEZM15}.
\begin{definition}\label{Def FF} \textbf{(Forward factorization)} $\{X_i\}_{i=1}^{\infty}$ is a sequence of random variables of forward factorization
if for any $n\geq 1$ and any bounded function $f(\cdot)$ of $X_n$ and any nonnegatively bounded function $g(\cdot )$ of $X_1,X_2,\cdots, X_{n-1}$, there is
$$\mathcal{E}[g(X_1,X_2,\cdots, X_{n-1})(f(X_n)-\mathcal{E}[f(X_n)])]\geq 0.$$
\end{definition}
Now we will show that negative association can be deduced from forward factorization, that is, negative association is a weaker condition than forward factorization.
\begin{proposition}
Forward factorization of $\{X_i\}_{i=1}^{\infty}$ implies that it is a sequence of negatively associated random variables.
\end{proposition}
\begin{proof}
Suppose that $\{f_i(\cdot)\}_{i=1}^{\infty}\subset C^+_{b}(\mathbb{R})$ with  the same monotonicity.
Then applying Proposition \ref{proposition 2.4}, there is
 \begin{align*}
 &\quad \ \mathbb{E}\big[\big(\prod_{i=1}^{n-1}f_i(X_i)\big)\cdot f_n(X_n)\big]
 -\mathbb{E}\big[\big(\prod_{i=1}^{n-1}f_i(X_i)\big)\cdot \mathbb{E} \big[f_n(X_n)\big]\big]\\
 &\leq \mathbb{E}\bigg[  \big(\prod_{i=1}^{n-1}f_i(X_i)\big)
 \cdot f_n(X_n)-\big(\prod_{i=1}^{n-1}f_i(X_i)\big)\cdot\mathbb{E} \big[f_n(X_n)\big]\bigg ]\\
 &\leq \mathbb{E}\bigg[  \big(\prod_{i=1}^{n-1}f_i(X_i)\big)
 \big(f_n(X_n)-\mathbb{E} \big[f_n(X_n)\big]\big )\bigg ]\\
&=-\mathcal{E}\big[  \big(\prod_{i=1}^{n-1}f_i(X_i)\big)
\big(-f_n(X_n)-\mathcal {E}\big[-f_n(X_n)\big]\big )\big ]\\
&=-\mathcal{E}\big[  \big(\prod_{i=1}^{n-1}f_i(X_i)\big)
\big(\bar{f}_n(X_n)-\mathcal {E}\big[\bar{f}_n(X_n)\big]\big )\big ]\\
&\leq 0,
 \end{align*}
where $\prod_{i=1}^{n-1}f_{i}(x_i)$ and $\bar{f}_{n}(x)=-f_n(x)$ satisfy the assumptions of $g(x_1,\cdots,x_{n-1})$ and $f(x)$  respectively in Definition \ref{Def FF}.

In addition, since the nonnegativity of $f_i(\cdot)$ implies the nonnegativity of $\mathbb{E} \big[f_i(X_i)\big]$, we have
\begin{align*}
&\quad \ \mathbb{E}\big[\big(\prod_{i=1}^{n-1}f_i(X_i)\big)\cdot f_n(X_n)\big]
 -\mathbb{E}\big[\prod_{i=1}^{n-1}f_i(X_i)\big]\cdot \mathbb{E} \big[f_n(X_n)\big]\\
& =\mathbb{E}\big[\big(\prod_{i=1}^{n-1}f_i(X_i)\big)\cdot f_n(X_n)\big]
 -\mathbb{E}\big[\big(\prod_{i=1}^{n-1}f_i(X_i)\big)\cdot \mathbb{E} \big[f_n(X_n)\big]\big]\\
& \leq 0.
\end{align*}
Therefore, $\{X_i\}_{i=1}^{\infty}$  is a sequence of negatively associated random variables.
\end{proof}
\begin{example}
Suppose that $\{P_j\}_{j\in J}$ is a family of linear probabilities and $\mathbb{V}(\cdot)=\sup_{j\in J}P_{j}(\cdot)$ and $v(\cdot)=\inf_{j\in J}P_{j}(\cdot)$ are the corresponding upper and lower probabilites. Let $X$ be a binomial distributed random variable satisfying
$$P_{j}(X=1)=p_j,\quad\ P_{j}(X=0)=1-p_j,$$
where $a<p_j< b$ and $0<a<b<1 $ for any $j\in J$.
Set $Y=-X$, then
$$P_{j}(Y=-1)=p_j,\quad\ P_{j}(Y=0)=1-p_j.$$
The random variables $X$ and $Y$  are negatively associated but not of forward factorization.
\end{example}
\begin{proof}
Set
$$ f(x)=\left\{
\begin{aligned}
1 \quad\quad &x\geq 1 \\
x \quad\quad & 0<x<1 \\
0 \quad\quad  & x\leq 0,
\end{aligned}
\right.
\quad\quad\quad\quad
g(x)=\left\{\begin{aligned}
1 \quad\quad & x\geq 0 \\
x+1 \quad & -1<x<0 \\
0 \quad\quad & x\leq -1.
\end{aligned}
\right.
$$
Therefore, $E_{P_j}[f(X)]=f(1)p_j+f(0)(1-p_j)=p_j$ yields that
$\mathcal{E}[f(X)]=\inf_{j\in J}E_{P_j}[f(X)]=\inf_{j\in I}p_j\overset{\triangle}{=}c$, where
$0<a\leq c\leq b<1$. Then there is
\begin{align*}
&\quad \ \mathcal{E}[g(Y)(f(X)-\mathcal{E}[f(X)])]\\
&=\inf_{j\in J}E_{P_j}[g(Y)(f(X)-c)]\\
&=\inf_{j\in J}\bigg\{   g(-1)(f(1)-c )p_j+g(0)(f(0)-c)(1-p_j) \bigg\}\\
&=\inf_{j\in J}\{ cp_j-c    \}\\
&=c^2-c<0.
\end{align*}
Therefore, random variables $X$ and $Y$ are not of forward factorization.

\end{proof}

In addition, our negative association is weaker than vertical independence  initiated by Chen et.al \cite{CZWP13}(Definition 2.4).

\begin{definition}(\textbf{Vertical independence})
Let $\{X_{i}\}_{i=1}^{\infty}$ be a sequence of  random variables on the upper probability space $(\Omega,\mathcal{F},\mathcal{P},\mathbb{V})$. $X_{n+1}$ is said to be vertically  independent  of
$(X_{1},\cdots,X_{n})$ under $\mathbb{E}[\cdot]$, if for each nonnegative measurable  function $f_{i}(\cdot)$, there is
$$\mathbb{E}\left[\prod_{i=1}^{n+1}f_{i}(X_{i})\right]
=\mathbb{E}\left[\prod_{i=1}^{n}f_{i}(X_{i})\right]
\mathbb{E}\left[f_{n+1}(X_{n+1})\right].$$
$\{X_{i}\}_{i=1}^{\infty}$ is said to be a sequence of vertically independent random variables, if $X_{n+1}$ is vertically independent of $(X_{1},X_{2},\cdots,X_{n})$ for each $n\in \mathbb{N}^{+}$.
\end{definition}

\begin{remark}
Since negative association changes the equation to inequality and only considers about nonnegatively continuous functions with the same monotonicity, it is obviously weaker than the condition of vertical independence.

\end{remark}

 \begin{lem}\label{NAlemma}
Suppose that $\{X_{i}\}_{i=1}^{\infty}$ is a sequence of negatively associated random variables on the upper probability space $(\Omega,\mathcal{F},\mathcal{P},\mathbb{V})$, and $\{f_{i}(x)\}_{i=1}^{\infty}$ is a sequence of  continuous functions with the same monotonicity. Then $\{f_{i}(X_{i})\}_{i=1}^{\infty}$ is also a sequence of negatively associated random variables.
\end{lem}
\begin{proof}
Without loss of generality, we only consider  the case that  $\{f_{i}(x)\}_{i=1}^{\infty}$ is a sequence of increasing and continuous functions. For any sequence of functions
$\{F_{i}(x)\}_{i=1}^{\infty} \subset C_b^+(\mathbb{R})$ with the same monotonicity, the composite functions $\{F_i\circ f_{i}(\cdot)\}_{i=1}^{\infty}$  are all in $C_b^+(\mathbb{R})$ with the same monotonicity.
Since $\{X_i\}_{i=1}^{\infty}$ is a sequence of negatively associated random variables, by Definition \ref{NA}, we obtain that
\begin{equation*}
\begin{split}
\mathbb{E}[\prod_{i=1}^{n+1}F_i(f_i(X_{i}))]
&=\mathbb{E}[\prod_{i=1}^{n+1}F_i\circ f_i(X_{i})]\\
&\leq\mathbb{E}[\prod_{i=1}^{n}F_i\circ f_i(X_{i})]\cdot
\mathbb{E}[F_{n+1}\circ f_{n+1}(X_{n+1})]\\
&= \mathbb{E}[\prod_{i=1}^{n}F_i(f_i(X_{i}))]
\cdot \mathbb{E}[F_{n+1}(f_{n+1}(X_{n+1}))] .
\end{split}
\end{equation*}
Therefore,  $\{f_{i}(X_{i})\}_{i=1}^{\infty}$ is also a sequence of negatively associated random variables.
\end{proof}

 \begin{lem}\label{exp}
Suppose that $\{X_{i}\}_{i=1}^{\infty}$ is a sequence of negatively associated random variables on the upper probability space $(\Omega,\mathcal{F},\mathcal{P},\mathbb{V})$, and $\{f_{i}(x)\}_{i=1}^{\infty}$ is a sequence of  bounded and continuous functions  with the same monotonicity. Then for any $n$, there is
$$\mathbb{E}[\exp\big\{\sum_{i=1}^{n}f_{i}(X_{i})\big \}]\leq \prod_{i=1}^{n}\mathbb{E}[e^{f_{i}(X_i)}].$$
\end{lem}
\begin{proof}
Without loss of generality, we only consider  the case that $\{f_{i}(x)\}_{i=1}^{\infty}$
is a sequence of nondecreasingly continuous and bounded functions. Set $F_{i}(x)=e^{f_i(\cdot)}$ for each $i\in\mathbb{Z}^+$. Since  $\{f_{i}(x)\}_{i=1}^{\infty}$ are all bounded, $F_{i}(x)=e^{f_i(\cdot)}$ is a sequence of nonnegatively bounded and continuous functions with the same monotonicity. Therefore, by Definition \ref{NA}, there is
\begin{align*}
\mathbb{E}[\exp\big\{\sum_{i=1}^{n}f_{i}(X_{i})\big \}]
&=\mathbb{E}[\prod_{i=1}^{n}e^{f_i(X_i)}]=\mathbb{E}[\prod_{i=1}^{n}F_i(X_i)]\\
&\leq \mathbb{E}[\prod_{i=1}^{n-1}F_{i}(X_{i})]\cdot\mathbb{E}[ F_{n}(X_n)]\\
&= \mathbb{E}[\prod_{i=1}^{n-1}F_{i}(X_{i}) \cdot F_{n-1}(X_{n-1})  ]\cdot\mathbb{E}[ F_{n}(X_n)] \\
&\leq \mathbb{E}[\prod_{i=1}^{n-2}F_{i}(X_{i})]\cdot\mathbb{E}[ F_{n-1}(X_{n-1})  ]\cdot\mathbb{E}[ F_{n}(X_n)]\\
&\leq \cdots\cdots\\
&\leq  \prod_{i=1}^{n}\mathbb{E}[F_{i}(X_i)]=\prod_{i=1}^{n}\mathbb{E}[e^{f_{i}(X_i)}].
\end{align*}

\end{proof}

\section{Strong limit theorems for weighted sums of negatively associated random variables}
In this section, we will prove our main result, the strong limit theorem for weighted sums of  negatively associated random variables on the upper probability space. Then  Kolmogorov type and Marcinkiewich-Zygmund type SLLNs for negatively associated random variables can be derived from it respectively.
To obtain the main result, we introduce the following lemma first.
\begin{lem}\label{mainlemma}
Let $\{X_{i}\}_{i=1}^{\infty}$ be a sequence of negatively associated  random variables on the upper probability space $(\Omega,\mathcal{F},\mathcal{P},\mathbb{V})$ such that
$\sup_{i\geq 1}\mathbb{E}[|X_{i}|^{\alpha+1}]<\infty$ for some constant $\alpha>0$. Let $\{a_{i}\}_{i=1}^{\infty}$ be a bounded sequence of positive numbers and $\{A_{n}\}_{n=1}^{\infty}$ be an increasing sequence of  positive numbers such that
\begin{align*}
\lim_{n\rightarrow \infty}\frac{A_{n}}{n^{\frac{1}{\beta+1}}}=\infty,\tag{4.1}\label{infty}
\end{align*}
where $\beta\in\left(0,\min(1,\alpha)\right)$ is a constant.
If there exists a constant $C$ such that
\begin{align*}
a_{n}|X_{n}-\mathbb{E}[X_{n}]|\leq C\frac{A_{n}}{\log(n+1)},\quad\quad n=1,2,\cdots . \tag{4.2}\label{condition}
\end{align*}
Then there exists a sufficiently large number $m>1$ such that
\begin{align*}
\sup_{n\geq 1}\mathbb{E}\left[\exp\left\{
\frac{m\log(n+1)}{A_{n}}\sum_{i=1}^{n}a_{i}(X_{i}-\mathbb{E}[X_{i}])
\right\}\right]<\infty.\tag{4.3}\label{bound}
\end{align*}
\end{lem}
\begin{proof}
It is obvious that $\sup_{i\geq 1}\mathbb{E}[|X_{i}|^{\alpha+1}]<\infty$ implies $\sup_{i\geq 1}\mathbb{E}[|X_{i}|]<\infty$ and $\sup_{i\geq 1}\mathbb{E}[|X_{i}-\mathbb{E}[X_{i}]|^{\alpha+1}]<\infty$.

For a given $\alpha>0$ and $0<\beta<\min(1,\alpha)$, we get $\frac{1}{\beta+1}>\frac{1}{\alpha+1}$. Therefore, by Equation (\ref{infty}) we get
$$\lim_{n\rightarrow\infty}
\frac{A_{n}}{n^{\frac{1}{\alpha+1}}\log(n+1)}=
\lim_{n\rightarrow\infty}\frac{A_{n}}{n^{\frac{1}{\beta+1}}}\cdot
\frac{n^{\frac{1}{\beta+1}}}{n^{\frac{1}{\alpha+1}}\log(n+1)}=\infty.$$
Then,
$$\lim_{n\rightarrow\infty}\frac{A_{n}^{\alpha+1}}{n(\log(n+1))^{\alpha+1}}
=\infty.$$
Thus, there exist a constant $n_{0}>0$ and a sufficient large number $m>1$ such that for $ n\geq n_{0},$
$$\frac{A_{n}^{\alpha+1}}{n(\log(n+1))^{\alpha+1}}\geq m^{\alpha+1},
\quad\quad\quad$$
that is
$$\frac{\left(m\log(n+1)\right)^{\alpha+1}}{A_{n}^{\alpha+1}}\leq\frac{1}{n},
\quad\quad\quad n\geq n_{0}.$$
In addition, it can be easily verified that for $0<\alpha\leq 1$,
$$e^{x}\leq 1+x+|x|^{\alpha+1}e^{2|x|},\quad\quad \forall x\in\mathbb{R}.$$
For any given $n\geq n_{0}$ and $i\leq n$, set $x=\frac{m\log(n+1)}{A_{n}}a_{i}(X_{i}-\mathbb{E}[X_{i}])$. Then we obtain that
\begin{equation*}
\begin{split}
&\quad\exp\left\{\frac{m\log(n+1)}{A_{n}}a_{i}(X_{i}-\mathbb{E}[X_{i}])\right\}\\
&\leq1+\frac{m\log(n+1)}{A_{n}}a_{i}(X_{i}-\mathbb{E}[X_{i}])\\
&+\frac{\left(m\log(n+1)\right)^{\alpha+1}}{A_{n}^{\alpha+1}}a_{i}^{\alpha+1}
|X_{i}-\mathbb{E}[X_{i}]|^{\alpha+1}
\exp\left\{2\frac{m\log(n+1)}{A_{n}}a_{i}|X_{i}-\mathbb{E}[X_{i}]|\right\}.
\end{split}
\end{equation*}
Since $\lim\limits_{n\rightarrow\infty}\frac{\log(n+1)}{A_{n}}=0$ and $\{A_n\}_{n=1}^{\infty}$ is an increasing sequence of positive numbers, for sufficiently large $n$, there is
\begin{equation*}\frac{m\log(n+1)}{A_{n}}a_{i}|X_{i}-\mathbb{E}[X_{i}]|
\leq \frac{m\log(i+1)}{A_{i}}a_{i}|X_{i}-\mathbb{E}[X_{i}]|\leq Cm. \tag{4.4}\label{bound2}
\end{equation*}
Therefore, we have
\begin{align*}
&\quad \ \exp\left\{\frac{m\log(n+1)}{A_{n}}a_{i}(X_{i}-\mathbb{E}[X_{i}])\right\}\\
&\leq1+\frac{m\log(n+1)}{A_{n}}a_{i}(X_{i}-\mathbb{E}[X_{i}])
+\frac{a_{i}^{\alpha+1}}{n}|X_{i}-\mathbb{E}[X_{i}]|^{\alpha+1}e^{2Cm}.
\end{align*}
Taking $\mathbb{E}[\cdot]$ on both sides of the above inequality and setting
$L:=\sup_{i\geq 1}a_{i}^{\alpha+1}\mathbb{E}[|X_{i}-\mathbb{E}[X_{i}]|^{\alpha+1}]$,
then
$$\mathbb{E}\left[ \exp\left\{\frac{m\log(n+1)}{A_{n}}a_{i}(X_{i}-\mathbb{E}[X_{i}])\right\} \right]\leq 1+\frac{L}{n}e^{2Cm}.
$$
Define
$$f_i(x)=\frac{m\log(n+1)}{A_{n}}\cdot a_{i}(x-\mathbb{E}[X_{i}])$$
and
\begin{align*}
\hat{f}_i(x)
&=\frac{m\log (n+1)}{A_n}\cdot a_i(x-\mathbb{E}[X_i])I_{\{a_i|x-\mathbb{E}[X_i]|
\leq \frac{CA_n}{\log(n+1)}\}}\\
&\quad  \ + Cm\cdot I_{\{a_i(x-\mathbb{E}[X_i])> \frac{CA_n}{\log(n+1)}\}}
- Cm\cdot I_{\{a_i(x-\mathbb{E}[X_i])< -\frac{CA_n}{\log(n+1)}\}}.
\end{align*}
Then $\{\hat{f}_i(x)\}_{i=1}^{\infty}$ is a sequence of bounded, continuous and increasing functions. Equation (\ref{bound2}) yields that for sufficiently large $n$, $a_i|x-\mathbb{E}[X_i]|
\leq \frac{CA_n}{\log(n+1)}$ always holds. Therefore, we obtain that for sufficiently large $n$, there is
$$f_i(X_i)=\hat{f}_i(X_i)\quad\quad\quad\quad\quad  i\leq n.$$
Consequently, applying Lemma \ref{exp}, we obtain
\begin{equation*}
\begin{split}
&\quad \ \mathbb{E}\bigg[\exp\left\{\frac{m\log(n+1)}{A_{n}}\sum_{i=1}^{n}a_{i}(X_{i}-\mathbb{E}[X_{i}])\right\}\bigg]\\
&=\mathbb{E}[e^{f_1(X_1)}e^{f_2(X_2)}\cdots e^{f_n(X_n)}]\\
&=\mathbb{E}[e^{\hat{f}_1(X_1)}e^{\hat{f}_2(X_2)}\cdots e^{\hat{f}_n(X_n)}]\\
&\leq \prod_{i=1}^{n}\mathbb{E}[ e^{\hat{f}_n(X_n)}]\\
&= \prod_{i=1}^{n}\mathbb{E}[ e^{f_n(X_n)}]\\
&=\prod_{i=1}^{n}\mathbb{E}\left[\exp\left\{\frac{m\log(n+1)}{A_{n}}a_{i}(X_{i}-\mathbb{E}[X_{i}])\right\}\right]\\
&\leq \left(1+\frac{L}{n}e^{2Cm}\right)^{n}\rightarrow e^{Le^{2Cm}}<\infty \quad as\quad n\rightarrow \infty.
\end{split}
\end{equation*}

For $\alpha>1$, choose $\gamma$ such that $\beta<\gamma<1$.
Since $\mathbb{E}[|X_{i}|^{\gamma+1}]
\leq \mathbb{E}[|X_{i}|^{\alpha+1}]<\infty$, replacing $\alpha$
by $\gamma$ in the above steps, we also get the desired result.
\end{proof}

\begin{thm}\label{mainthm}
Let $\{X_{i}\}_{i=1}^{\infty}$ be a sequence of negatively associated random variables on the upper  probability space $(\Omega,\mathcal{F},\mathcal{P},\mathbb{V})$ such that
$\sup_{i\geq 1}\mathbb{E}[|X_{i}|^{\alpha+1}]<\infty$ for some constant $\alpha>0$. Let $\{a_{i}\}_{i=1}^{\infty}$ be a bounded sequence of positive numbers  and $\{A_{n}\}_{n=1}^{\infty}$  be an increasing sequence of positive numbers such that
$$\lim_{n\rightarrow \infty}\frac{A_{n}}{n^{\frac{1}{\beta+1}}}=\infty,$$
where $\beta$ is a constant such that $\beta\in\left(0,\min(1,\alpha)\right)$.
Then,
\[\mathbb{V}\left(\left\{\liminf_{n\rightarrow \infty}
\frac{\sum_{i=1}^{n}a_i(X_{i}-\mathcal{E}[X_{i}])}{A_{n}}<0\right\}
\bigcup\left\{\limsup_{n\rightarrow \infty}
\frac{\sum_{i=1}^{n}a_i(X_{i}-\mathbb{E}[X_{i}])}{A_{n}}>0\right\}\right)=0,
\tag{4.5}\label{thm}\]
also
$$v\left(\liminf_{n\rightarrow \infty}
\frac{\sum_{i=1}^{n}a_i(X_{i}-\mathcal{E}[X_{i}])}{A_{n}}\geq0\right)=1,
\quad\quad
v\left(\limsup_{n\rightarrow \infty}
\frac{\sum_{i=1}^{n}a_i(X_{i}-\mathbb{E}[X_{i}])}{A_{n}}\leq0\right)=1.
$$
\end{thm}

\begin{proof}
Denote $A:=\left\{\omega\bigg| \limsup_{n\rightarrow\infty}
\frac{\sum_{i=1}^{n}a_i(X_{i}(\omega)-\mathbb{E}[X_{i}])}{A_{n}}>0\right\}$

\vspace{10pt}
\ \ \ \ \ and $B:=\left\{\omega\bigg| \liminf_{n\rightarrow\infty}
\frac{\sum_{i=1}^{n}a_i(X_{i}(\omega)-\mathbb{E}[X_{i}])}{A_{n}}<0\right\}$.

\vspace{10pt}\noindent Since $\max\{\mathbb{V}(A),\mathbb{V}(B)\}\leq \mathbb{V}(A\bigcup B)
\leq \mathbb{V}(A)+\mathbb{V}(B)$, it is obvious that Equation (\ref{thm})
is equivalent to the conjunction of
\begin{align*}
\mathbb{V}\left( \limsup_{n\rightarrow\infty}
\frac{\sum_{i=1}^{n}a_i(X_{i}-\mathbb{E}[X_{i}])}{A_{n}}>0\right)=0,
\tag{4.6}\label{aim}
\end{align*}
and
\begin{align*}
\mathbb{V}\left( \liminf_{n\rightarrow\infty}
\frac{\sum_{i=1}^{n}a_i(X_{i}-\mathcal{E}[X_{i}])}{A_{n}}<0\right)=0.
\tag{4.7}\label{aim2}
\end{align*}
If Equation (\ref{aim}) holds, considering the sequence $\{-X_{i}\}_{i=1}^{\infty} $, we get
$$\mathbb{V}\left( \limsup_{n\rightarrow\infty}
\frac{\sum_{i=1}^{n}a_i(-X_{i}-\mathbb{E}[-X_{i}])}{A_{n}}>0\right)=0.$$
Since $\mathbb{E}[-X_{i}]=-\mathcal{E}[X_{i}]$, the equation above is
equivalent to Equation (\ref{aim2}).
Therefore, we only need to prove Equation (\ref{aim}).

The proof is completed through two steps.

\noindent \textbf{Step 1:}
Assume that there exists a constant $C>0$ such that $a_i|X_{i}-\mathbb{E}[X_{i}]|\leq \frac{CA_{i}}{\log(i+1)}$ for any $i\in\mathbb{N^{+}} $. Then, $\{X_{i}\}_{i=1}^{\infty}$ satisfies all the assumptions of
Lemma \ref{mainlemma}.
To prove Equation (\ref{aim}), it is sufficient to prove that for any $\epsilon >0$,
\begin{align*}
\mathbb{V}\left(\bigcap_{m=1}^{\infty}\bigcup_{n=m}^{\infty}
\left\{   \frac{\sum_{i=1}^{n}a_i(X_{i}-\mathbb{E}[X_{i}])}{A_{n}}\geq\epsilon
\right\}
\right)=0.\tag{4.8} \label{limsup}
\end{align*}
By Lemma \ref{mainlemma}, for $\epsilon>0$, choose $m>\frac{1}{\epsilon}$ such that
$$\sup_{n\geq 1}\mathbb{E}\left[\exp\left\{
\frac{m\log(n+1)}{A_{n}}\sum_{i=1}^{n}a_i(X_{i}-\mathbb{E}[X_{i}])
\right\}\right]<\infty.$$
By Proposition \ref{ineq} (Chebyshev's inequality), there is
\begin{align*}
&\quad \ \mathbb{V}\left(  \frac{\sum_{i=1}^{n}a_i(X_{i}-\mathbb{E}[X_{i}])}{A_{n}}\geq\epsilon\right)\\   
&=\mathbb{V}\left(  \frac{m\log(n+1)}{A_{n}}\sum_{i=1}^{n}a_i(X_{i}-\mathbb{E}[X_{i}])\geq
\epsilon m\log(n+1)\right)\\
&\leq e^{-\epsilon m\log(n+1)}\mathbb{E}\left[\exp\left\{
\frac{m\log(n+1)}{A_{n}}\sum_{i=1}^{n}a_i(X_{i}-\mathbb{E}[X_{i}])
\right\}\right]\\
&\leq \frac{1}{(n+1)^{m\epsilon}}\cdot\sup_{n\geq 1}\mathbb{E}\left[\exp\left\{
\frac{m\log(n+1)}{A_{n}}\sum_{i=1}^{n}a_i(X_{i}-\mathbb{E}[X_{i}])
\right\}\right].
\end{align*}
Since $m\epsilon>1$ and $\sup_{n\geq 1}\mathbb{E}\left[\exp\left\{
\frac{m\log(n+1)}{A_{n}}\sum_{i=1}^{n}a_i(X_{i}-\mathbb{E}[X_{i}])
\right\}\right]<\infty$, following the convergence of
$\sum_{i=1}^{\infty}\frac{1}{(n+1)^{m\epsilon}}<\infty$, we get
$$\sum_{n=1}^{\infty}\mathbb{V}\left(  \frac{\sum_{i=1}^{n}a_i(X_{i}-\mathbb{E}[X_{i}])}{A_{n}}\geq\epsilon\right) <\infty.$$
By Lemma \ref{BClemma} (Borel-Cantelli Lemma), we obtain that for any $\epsilon>0$, Equation (\ref{limsup}) holds. That is for any $\epsilon>0$, there is
$$ \mathbb{V}\left(  \limsup_{n\rightarrow \infty}
\frac{\sum_{i=1}^{n}a_i(X_{i}-\mathbb{E}[X_{i}])}{A_{n}}\geq\epsilon
\right)=0.$$
Then by the continuity from below of $\mathbb{V}$, we achieve
$$ \mathbb{V}\left(  \limsup_{n\rightarrow \infty}
\frac{\sum_{i=1}^{n}a_i(X_{i}-\mathbb{E}[X_{i}])}{A_{n}}>0
\right)=0.$$

\textbf{Step 2:}
Without the assumption in Step 1, for a given constant $C>0$, set
\begin{align*}
Y_i&=(X_{i}-\mathbb{E}[X_{i}])
I_{\{|X_{i}-\mathbb{E}[X_{i}]|\leq \frac{CA_{i}}{a_i\log(i+1)}\}}
+\frac{CA_{i}}{a_i\log(i+1)}I_{\{X_{i}-\mathbb{E}[X_{i}]> \frac{CA_{i}}{a_i\log(i+1)}\}}\\
&\quad \ -\frac{CA_{i}}{a_i\log(i+1)}I_{\{X_{i}-\mathbb{E}[X_{i}]<- \frac{CA_{i}}{a_i\log(i+1)}\}}
-\mathbb{E}\left[  (X_{i}-\mathbb{E}[X_{i}])
I_{\{|X_{i}-\mathbb{E}[X_{i}]|\leq \frac{CA_{i}}{a_i\log(i+1)}\}} \right. \\
&\quad \ \left.+\frac{CA_{i}}{a_i\log(i+1)}I_{\{X_{i}-\mathbb{E}[X_{i}]> \frac{CA_{i}}{a_i\log(i+1)}\}}
-\frac{CA_{i}}{a_i\log(i+1)}I_{\{X_{i}-\mathbb{E}[X_{i}]<- \frac{CA_{i}}{a_i\log(i+1)}\}}   \right]\\
&\quad \ +\mathbb{E}[X_i].    \tag{4.9}\label{Y}
\end{align*}
Define $\{f_i(x)\}_{i=1}^{\infty}$ by
$$f_i(x)=(x-b_i)I_{\{|x-b_i|\leq c_i\}}+c_iI_{\{x-b_i>c_i\}}-c_iI_{\{x-b_i<-c_i\}}+d_i,$$
where $b_i,c_i,d_i$ are all constants. Then $\{f_i(x)\}_{i=1}^{\infty}$ is a sequence of  bounded, continuous and increasing functions of $x$.

For given $i$, denote

\vspace{10pt}\quad\quad\quad\quad $b_i=\mathbb{E}[X_i],$

\vspace{10pt} \quad\quad\quad\quad $c_i=\frac{CA_{i}}{a_i\log(i+1)},$
\begin{align*}
d_i&=\mathbb{E}[X_i]-\mathbb{E}\left[  (X_{i}-\mathbb{E}[X_{i}])
I_{\{|X_{i}-\mathbb{E}[X_{i}]|\leq \frac{CA_{i}}{a_i\log(i+1)}\}}   \right.\\
&\quad \quad \left. +\frac{CA_{i}}{a_i\log(i+1)}I_{\{X_{i}-\mathbb{E}[X_{i}]> \frac{CA_{i}}{a_i\log(i+1)}\}}
-\frac{CA_{i}}{a_i\log(i+1)}I_{\{X_{i}-\mathbb{E}[X_{i}]<- \frac{CA_{i}}{a_i\log(i+1)}\}}   \right].
\end{align*}
Then we obtain that $Y_i=f_i(X_i)$.  Since $\{f_i(x)\}_{i=1}^{\infty}$ is a sequence
of increasing and continuous functions, applying Lemma \ref{NAlemma}, we achieve that $\{Y_i\}_{i=1}^{\infty}$ is a sequence of negatively associated random variables.
Now, we prove that $\{Y_i\}_{i=1}^{\infty}$ satisfies all the assumptions in Step 1.

It is obvious that $\mathbb{E}[Y_i]=\mathbb{E}[X_i]$. In addition,
\begin{align*}
&\quad \ |Y_i-\mathbb{E}[Y_i]|\\
&=|Y_i-\mathbb{E}[X_i]|\\
&=\bigg|(X_{i}-\mathbb{E}[X_{i}])I_{\{|X_{i}-\mathbb{E}[X_{i}]|\leq c_i\}}
+c_i I_{\{X_{i}-\mathbb{E}[X_{i}]> c_i\}}
-c_i I_{\{X_{i}-\mathbb{E}[X_{i}]<- c_i\}}\\
&\quad \ -\mathbb{E}\left[  (X_{i}-\mathbb{E}[X_{i}])I_{\{|X_{i}-\mathbb{E}[X_{i}]|\leq c_i\}}
+c_i I_{\{X_{i}-\mathbb{E}[X_{i}]> c_i\}}
-c_i I_{\{X_{i}-\mathbb{E}[X_{i}]<- c_i\}}   \right]
\bigg|\\  \displaybreak
&\leq |(X_{i}-\mathbb{E}[X_{i}])I_{\{|X_{i}-\mathbb{E}[X_{i}]|\leq c_i\}}|
+|c_i I_{\{X_{i}-\mathbb{E}[X_{i}]> c_i\}}|
+|c_i I_{\{X_{i}-\mathbb{E}[X_{i}]<- c_i\}}|\\
&\quad\ +\mathbb{E}\big[\big| (X_{i}-\mathbb{E}[X_{i}])I_{\{|X_{i}-\mathbb{E}[X_{i}]|\leq c_i\}} \big|\big]
+\mathbb{E}\big[\big|  c_i I_{\{X_{i}-\mathbb{E}[X_{i}]> c_i\}}  \big| \big]
+\mathbb{E}\big[\big|  c_i I_{\{X_{i}-\mathbb{E}[X_{i}]<- c_i\}}   \big| \big]\\
&\leq c_i+c_i+c_i+c_i+c_i+c_i\\
&=\frac{6CA_{i}}{a_i\log(i+1)},
\end{align*}
that is, there exists a constant $C'=6C$ such that $a_i|Y_{i}-\mathbb{E}[Y_{i}]|\leq \frac{C'A_{i}}{\log(i+1)}$ for any $i\in\mathbb{N^{+}} $.
Furthermore,
\begin{align*}
&\quad \ \mathbb{E}[|Y_i-\mathbb{E}[Y_i]|^{\alpha+1}]\\
&=\mathbb{E}\bigg[\bigg|(X_{i}-\mathbb{E}[X_{i}])I_{\{|X_{i}-\mathbb{E}[X_{i}]|\leq c_i\}}
+c_i I_{\{X_{i}-\mathbb{E}[X_{i}]> c_i\}}
-c_i I_{\{X_{i}-\mathbb{E}[X_{i}]<- c_i\}}\\
&\quad -\mathbb{E}\big[(X_{i}-\mathbb{E}[X_{i}])I_{\{|X_{i}-\mathbb{E}[X_{i}]|\leq c_i\}}
+c_i I_{\{X_{i}-\mathbb{E}[X_{i}]> c_i\}}
-c_i I_{\{X_{i}-\mathbb{E}[X_{i}]<- c_i\}}
\big]\bigg|^{\alpha+1}\bigg]\\
&\leq\mathbb{E}\bigg[\bigg(|X_{i}-\mathbb{E}[X_{i}]|I_{\{|X_{i}-\mathbb{E}[X_{i}]|\leq c_i\}}
+c_i I_{\{X_{i}-\mathbb{E}[X_{i}]> c_i\}}
+c_i I_{\{X_{i}-\mathbb{E}[X_{i}]<- c_i\}}\\
&\quad +\mathbb{E}\big[|X_{i}-\mathbb{E}[X_{i}]|I_{\{|X_{i}-\mathbb{E}[X_{i}]|\leq c_i\}}
+c_i I_{\{X_{i}-\mathbb{E}[X_{i}]> c_i\}}
+c_i I_{\{X_{i}-\mathbb{E}[X_{i}]<- c_i\}}
\big]\bigg)^{\alpha+1}\bigg]\\
&=\mathbb{E}\bigg[\bigg(|X_{i}-\mathbb{E}[X_{i}]|I_{\{|X_{i}-\mathbb{E}[X_{i}]|\leq c_i\}}
+c_i I_{\{|X_{i}-\mathbb{E}[X_{i}]|> c_i\}}\\
&\quad +\mathbb{E}\big[|X_{i}-\mathbb{E}[X_{i}]|I_{\{|X_{i}-\mathbb{E}[X_{i}]|\leq c_i\}}
+c_i I_{|\{X_{i}-\mathbb{E}[X_{i}]|> c_i\}}
\big]\bigg)^{\alpha+1}\bigg]\\
&\leq\mathbb{E}\bigg[\bigg(|X_{i}-\mathbb{E}[X_{i}]|I_{\{|X_{i}-\mathbb{E}[X_{i}]|\leq c_i\}}
+|X_{i}-\mathbb{E}[X_{i}]| I_{\{|X_{i}-\mathbb{E}[X_{i}]|> c_i\}}\\
&\quad +\mathbb{E}\big[|X_{i}-\mathbb{E}[X_{i}]|I_{\{|X_{i}-\mathbb{E}[X_{i}]|\leq c_i\}}
+|X_{i}-\mathbb{E}[X_{i}]|I_{|\{X_{i}-\mathbb{E}[X_{i}]|> c_i\}}
\big]\bigg)^{\alpha+1}\bigg]\\
&=\mathbb{E}\bigg[ \big(|X_i-\mathbb{E}[X_i]|+\mathbb{E}[|X_i-\mathbb{E}[X_i]|]\big)^{\alpha+1}   \bigg]\\
&<\infty.
\end{align*}
Therefore, $\{Y_{i}\}_{i=1}^{\infty}$ satisfies all the assumptions in Lemma \ref{mainlemma}. Then by Step 1, we obtain that
$$\mathbb{V}\left(\limsup_{n\rightarrow \infty}
\frac{\sum_{i=1}^{n}a_i(Y_{i}-\mathbb{E}[Y_{i}])}{A_{n}}>0\right)=0,$$
that is
\begin{align*}
\limsup_{n\rightarrow \infty}
\frac{\sum_{i=1}^{n}a_i(Y_{i}-\mathbb{E}[Y_{i}])}{A_{n}}\leq0.\quad\quad q.s.\tag{4.10}\label{qs}
\end{align*}
Equation \ref{Y} and $\mathbb{E}[X_i]=\mathbb{E}[Y_i]$ yield that
\begin{equation*}
\begin{split}
&\quad \  X_i-\mathbb{E}[X_i]-\big(Y_i-\mathbb{E}[Y_i] \big)\\
&=X_i-\mathbb{E}[X_i]-\big(Y_i-\mathbb{E}[X_i] \big)\\
&=X_{i}-\mathbb{E}[X_i]-(X_{i}-\mathbb{E}[X_{i}])
I_{\{|X_{i}-\mathbb{E}[X_{i}]|\leq c_i\}}
 -c_i I_{\{X_{i}-\mathbb{E}[X_{i}]>c_i\}} +c_i
I_{\{X_{i}-\mathbb{E}[X_{i}]<- c_i\}}\\
&\quad\ +\mathbb{E}\left[  (X_{i}-\mathbb{E}[X_{i}])
I_{\{|X_{i}-\mathbb{E}[X_{i}]|\leq c_i\}}
+c_i I_{\{X_{i}-\mathbb{E}[X_{i}]> c_i\}}
-c_i I_{\{X_{i}-\mathbb{E}[X_{i}]<- c_i\}}   \right]\\
&=(X_{i}-\mathbb{E}[X_{i}])
I_{\{|X_{i}-\mathbb{E}[X_{i}]|> c_i\}} -c_i I_{\{X_{i}-\mathbb{E}[X_{i}]> c_i\}} +c_i I_{\{X_{i}-\mathbb{E}[X_{i}]<- c_i\}}\\
&\quad\ +\mathbb{E}\left[  (X_{i}-\mathbb{E}[X_{i}])
I_{\{|X_{i}-\mathbb{E}[X_{i}]|\leq c_i\}}
+c_i I_{\{X_{i}-\mathbb{E}[X_{i}]> c_i\}}
-c_i I_{\{X_{i}-\mathbb{E}[X_{i}]<- c_i\}}   \right].
\end{split}
\end{equation*}
Therefore, we have
\begin{align*}
&\quad\ \frac{1}{A_n}\sum_{i=1}^{n}a_i(X_{i}-\mathbb{E}[X_i])\\
&=\frac{1}{A_n}\sum_{i=1}^{n} a_i(Y_{i}-\mathbb{E}[Y_i])
+\frac{1}{A_n}\sum_{i=1}^{n} a_i(X_{i}-\mathbb{E}[X_i])I_{\{|X_{i}-\mathbb{E}[X_{i}]|> c_i\}}\\
&\quad \ -\frac{1}{A_n}\sum_{i=1}^{n}a_i c_i I_{\{X_{i}-\mathbb{E}[X_{i}]> c_i\}}
+\frac{1}{A_n}\sum_{i=1}^{n}a_i c_i I_{\{X_{i}-\mathbb{E}[X_{i}]<- c_i\}}\\
&\quad\ +\frac{1}{A_n}\sum_{i=1}^{n}\mathbb{E}\left[  a_i(X_{i}-\mathbb{E}[X_{i}])
I_{\{|X_{i}-\mathbb{E}[X_{i}]|\leq c_i\}}
+a_i  c_i  I_{\{X_{i}-\mathbb{E}[X_{i}]> c_i \}}
 -a_i c_i I_{\{X_{i}-\mathbb{E}[X_{i}]<- c_i\}}   \right]\\
&\leq\frac{1}{A_n}\sum_{i=1}^{n}a_i (Y_{i}-\mathbb{E}[Y_i])
+\frac{1}{A_n}\sum_{i=1}^{n} a_i|X_{i}-\mathbb{E}[X_i]|I_{\{|X_{i}-\mathbb{E}[X_{i}]|> c_i\}}
 +\frac{1}{A_n}\sum_{i=1}^{n}a_i c_i I_{\{X_{i}-\mathbb{E}[X_{i}]> c_i\}}\\
&\quad \ +\frac{1}{A_n}\sum_{i=1}^{n}a_i c_i I_{\{X_{i}-\mathbb{E}[X_{i}]<- c_i\}}
 +\frac{1}{A_n}\sum_{i=1}^{n}\mathbb{E}\left[ a_i (X_{i}-\mathbb{E}[X_{i}])
I_{\{|X_{i}-\mathbb{E}[X_{i}]|\leq c_i\}}\right]\\
&\quad\ +\frac{1}{A_n}\sum_{i=1}^{n}\mathbb{E}\left[a_i c_i I_{\{X_{i}-\mathbb{E}[X_{i}]> c_i\}}\right]  +\frac{1}{A_n}\sum_{i=1}^{n}\mathbb{E}\left[a_i c_i I_{\{X_{i}-\mathbb{E}[X_{i}]<- c_i \}}   \right]\\
&= \frac{1}{A_n}\sum_{i=1}^{n}a_i (Y_{i}-\mathbb{E}[Y_i])
+\frac{1}{A_n}\sum_{i=1}^{n} a_i|X_{i}-\mathbb{E}[X_i]|I_{\{a_i |X_{i}-\mathbb{E}[X_{i}]|> \frac{CA_{i}}{\log(i+1)}\}}\\
&\quad \ +\frac{1}{A_n}\sum_{i=1}^{n}\frac{CA_{i}}{\log(i+1)}I_{\{a_i(X_{i}-\mathbb{E}[X_{i}])> \frac{CA_{i}}{\log(i+1)}\}}
+\frac{1}{A_n}\sum_{i=1}^{n}\frac{CA_{i}}{\log(i+1)}I_{\{a_i(X_{i}-\mathbb{E}[X_{i}])<- \frac{CA_{i}}{\log(i+1)}\}}\\
&\quad \ +\frac{1}{A_n}\sum_{i=1}^{n}\mathbb{E}\left[ a_i (X_{i}-\mathbb{E}[X_{i}])
I_{\{a_i|X_{i}-\mathbb{E}[X_{i}]|\leq \frac{CA_{i}}{\log(i+1)}\}}\right]\\
&\quad\ +\frac{1}{A_n}\sum_{i=1}^{n}\mathbb{E}\left[\frac{CA_{i}}{\log(i+1)} I_{\{a_i (X_{i}-\mathbb{E}[X_{i}])> \frac{CA_{i}}{\log(i+1)}\}}\right]\\
&\quad \   +\frac{1}{A_n}\sum_{i=1}^{n}\mathbb{E}\left[\frac{CA_{i}}{\log(i+1)} I_{\{a_i(X_{i}-\mathbb{E}[X_{i}])<- \frac{CA_{i}}{\log(i+1)} \}}   \right].
\end{align*}
Consider that
\begin{align*}
&\quad \mathbb{E}\left[a_i(X_{i}-\mathbb{E}[X_{i}])
I_{\{a_{i}|X_{i}-\mathbb{E}[X_{i}]|\leq \frac{CA_{i}}{\log(i+1)}\}}
\right]\\
&=\mathbb{E}\left[a_i(X_{i}-\mathbb{E}[X_{i}])+a_i(\mathbb{E}[X_{i}]-X_{i})
I_{\{a_{i}|X_{i}-\mathbb{E}[X_{i}]|>\frac{CA_{i}}{\log(i+1)}\}}
\right]\\
&\leq a_i\mathbb{E}\left[X_{i}-\mathbb{E}[X_{i}]\right]+
a_i\mathbb{E}\left[(\mathbb{E}[X_{i}]-X_{i})
I_{\{a_{i}|X_{i}-\mathbb{E}[X_{i}]|>\frac{CA_{i}}{\log(i+1)}\}}
\right]\\
&\leq a_i\mathbb{E}\left[|X_{i}-\mathbb{E}[X_{i}]|
I_{\{a_{i}|X_{i}-\mathbb{E}[X_{i}]|>\frac{CA_{i}}{\log(i+1)}\}}
\right],
\end{align*}
For convenience, denote
$$H(n)=   \frac{1}{A_n}\sum_{i=1}^{n}a_i |X_{i}-\mathbb{E}[X_i]|I_{\{a_i|X_{i}-\mathbb{E}[X_{i}]|> \frac{CA_{i}}{\log(i+1)}\}}, \ \ \ \ \ $$
$$J(n)=  \frac{1}{A_n}\sum_{i=1}^{n}\frac{CA_{i}}{\log(i+1)}I_{\{a_i(X _{i}-\mathbb{E}[X_{i}])> \frac{CA_{i}}{\log(i+1)}\}}, \ \ \ \ \ \ \ \ \ \ \ $$
$$K(n)=  \frac{1}{A_n}\sum_{i=1}^{n}\frac{CA_{i}}{\log(i+1)}I_{\{a_i(X_{i}-\mathbb{E}[X_{i}])<- \frac{CA_{i}}{\log(i+1)}\}},  \ \ \ \ \ \ \ $$
$$L(n)=  \frac{1}{A_n}\sum_{i=1}^{n}a_i\mathbb{E}\left[  |X_{i}-\mathbb{E}[X_{i}]|
I_{\{a_i|X_{i}-\mathbb{E}[X_{i}]|>\frac{CA_{i}}{\log(i+1)}\}}\right] , $$
$$M(n)=   \frac{1}{A_n}\sum_{i=1}^{n}\mathbb{E}\left[\frac{CA_{i}}{\log(i+1)}I_{\{a_i(X_{i}-\mathbb{E}[X_{i}])> \frac{CA_{i}}{\log(i+1)}\}}\right], \ \ \ $$
$$N(n)=  \frac{1}{A_n}\sum_{i=1}^{n}\mathbb{E}\left[\frac{CA_{i}}{\log(i+1)}I_{\{a_i(X_{i}-\mathbb{E}[X_{i}])<- \frac{CA_{i}}{\log(i+1)}\}}   \right]. \ \ \ $$
Then we can conclude that
\begin{align*}
&\quad \  \frac{1}{A_n}\sum_{i=1}^{n}a_i (X_{i}-\mathbb{E}[X_i])\\
&  \leq \frac{1}{A_n}\sum_{i=1}^{n}a_i (Y_{i}-\mathbb{E}[Y_i])
 +H(n)+J(n)+K(n)+L(n)+M(n)+N(n).\tag{4.11}\label{ieq}
\end{align*}
Consider that
\begin{align*}
&\quad \  \sum_{i=1}^{n}\frac{a_{i}\mathbb{E}[|X_{i}-\mathbb{E}[X_{i}]|
I_{\{a_{i}|X_{i}-\mathbb{E}[X_{i}]|>\frac{CA_{i}}{\log(i+1)}\}}
]}{A_{i}}\\
&=\sum_{i=1}^{n}\frac{a_{i}}{A_{i}}\mathbb{E}[
|X_{i}-\mathbb{E}[X_{i}]|^{\alpha+1}|X_{i}-\mathbb{E}[X_{i}]|^{-\alpha}
I_{\{a_{i}|X_{i}-\mathbb{E}[X_{i}]|>\frac{CA_{i}}{\log(i+1)}\}}]\\    \displaybreak
&\leq \sum_{i=1}^{n}\frac{a_{i}}{A_{i}}\mathbb{E}\bigg[
|X_{i}-\mathbb{E}[X_{i}]|^{\alpha+1}\big(\frac{CA_{i}}{a_i\log(i+1)}\big)^{-\alpha}
I_{\{a_{i}|X_{i}-\mathbb{E}[X_{i}]|>\frac{CA_{i}}{\log(i+1)}\}}\bigg]\\ 
&=\sum_{i=1}^{n}\frac{a_{i}}{A_{i}}\frac{a_{i}^{\alpha}
(\log(i+1))^{\alpha}}{C^{\alpha}A_{i}^{\alpha}}
\mathbb{E}[|X_{i}-\mathbb{E}[X_{i}]|^{\alpha+1}I_{\{a_{i}|X_{i}-\mathbb{E}[X_{i}]|>\frac{CA_{i}}{\log(i+1)}\}}]\\
&\leq\sum_{i=1}^{n}\frac{a_{i}}{A_{i}}\frac{a_{i}^{\alpha}
(\log(i+1))^{\alpha}}{C^{\alpha}A_{i}^{\alpha}}
\mathbb{E}[|X_{i}-\mathbb{E}[X_{i}]|^{\alpha+1}]\\
&\leq\sup_{i\geq1}\bigg\{\frac{a_{i}^{\alpha+1}}{C^{\alpha}}
\mathbb{E}[|X_{i}-\mathbb{E}[X_{i}]|^{\alpha+1}]\bigg\}
\sum_{i=1}^{n}\frac{(\log(i+1))^{\alpha}}{A_{i}^{\alpha+1}}.
\end{align*}
Following $\lim_{n\rightarrow\infty}\frac{A_{n}}{n^{\frac{1}{\beta+1}}}=\infty$
and $\alpha>\beta$,
we get
$$\lim_{n\rightarrow\infty}\sum_{i=1}^{n}\frac{(\log(i+1))^{\alpha}}{A_{i}^{\alpha+1}}
\leq \lim_{n\rightarrow\infty}\sum_{i=1}^{n}
\frac{(\log(i+1))^{\alpha}}{i^{\frac{\alpha+1}{\beta+1}}}<\infty.$$
Then, we obtain that
$$\sum_{i=1}^{\infty}\frac{a_{i}\mathbb{E}[|X_{i}-\mathbb{E}[X_{i}]|
I_{\{a_{i}|X_{i}-\mathbb{E}[X_{i}]|>\frac{CA_{i}}{\log(i+1)}\}}
]}{A_{i}}<\infty.$$
Applying Kronecker's Lemma, we can achieve that
$\lim_{n\rightarrow\infty}L(n)=0$.

\noindent In addition, since
\begin{align*}
&\quad\mathbb{E}\left[\sum_{i=1}^{\infty}\frac{a_{i}|X_{i}-\mathbb{E}[X_{i}]|
I_{\{a_{i}|X_{i}-\mathbb{E}[X_{i}]|>\frac{CA_{i}}{\log(i+1)}\}}
}{A_{i}}\right]\\
&\leq \sum_{i=1}^{\infty}\frac{a_{i}\mathbb{E}[|X_{i}-\mathbb{E}[X_{i}]|
I_{\{a_{i}|X_{i}-\mathbb{E}[X_{i}]|>\frac{CA_{i}}{\log(i+1)}\}}
]}{A_{i}}\\
&<\infty,
\end{align*}
we get
$\sum_{i=1}^{\infty}\frac{a_{i}|X_{i}-\mathbb{E}[X_{i}]|
I_{\{a_{i}|X_{i}-\mathbb{E}[X_{i}]|>\frac{CA_{i}}{\log(i+1)}\}}
}{A_{i}}<\infty$ q.s.
Then applying Kronecker's Lemma again, we achieve that
$\lim_{n\rightarrow\infty}H(n)=0$ q.s.

Now,  we focus on $M(n)$. Since
\begin{align*} &\quad \ \sum_{i=1}^{\infty}\frac{1}{A_i}\mathbb{E}\left[\frac{CA_{i}}{\log(i+1)}I_{\{X_{i}-\mathbb{E}[X_{i}]> \frac{CA_{i}}{a_{i}\log(i+1)}\}}\right]\\
&=\sum_{i=1}^{\infty}\frac{C}{\log(i+1)}\mathbb{E}\left[I_{\{X_{i}-\mathbb{E}[X_{i}]> \frac{CA_{i}}{a_{i}\log(i+1)}\}}\right]\\
&\leq\sum_{i=1}^{\infty}\frac{C}{\log(i+1)}
\mathbb{E}\left[\frac{|X_i-\mathbb{E}[X_i]|^{\alpha+1}}{(\frac{CA_{i}}{a_{i}\log(i+1)})^{\alpha+1}}I_{\{X_{i}-\mathbb{E}[X_{i}]> \frac{CA_{i}}{a_{i}\log(i+1)}\}}\right]\\     \displaybreak
&\leq\sum_{i=1}^{\infty}\frac{C}{\log(i+1)}
\frac{a_i^{\alpha+1}(\log(i+1))^{\alpha+1}}{C^{\alpha+1}A_i^{\alpha+1}}
\mathbb{E}\left[|X_i-\mathbb{E}[X_i]|^{\alpha+1}\right]\\
&=\frac{1}{C^{\alpha}}\sup_{i\geq 1}\{a_i^{\alpha+1}\mathbb{E}[|X_i-\mathbb{E}[X_i]|^{\alpha+1}]\}
\sum_{i=1}^n\frac{(\log(i+1))^{\alpha}}{A_i^{\alpha+1}}\\
&<\infty,
\end{align*}
then by Kronecker's Lemma, we conclude that $\lim_{n\rightarrow\infty}M(n)=0$.
Similarly, we achieve that
$$\lim_{n\rightarrow\infty}N(n)=0,$$
$$\lim_{n\rightarrow\infty}J(n)=0,\quad q.s.$$
$$\lim_{n\rightarrow\infty}K(n)=0. \quad q.s.$$
Together with Inequality (\ref{ieq}), we get
$$\limsup_{n\rightarrow\infty}\frac{1}{A_{n}}
\sum_{i=1}^{n}a_{i}(X_{i}-\mathbb{E}[X_{i}])
\leq\limsup_{n\rightarrow\infty}\frac{1}{A_{n}}
\sum_{i=1}^{n}a_{i}(Y_{i}-\mathbb{E}[Y_{i}]),     \quad\quad q.s. $$
Thus, applying Inequality (\ref{qs}), we achieve
$$\limsup_{n\rightarrow\infty}\frac{1}{A_{n}}
\sum_{i=1}^{n}a_{i}(X_{i}-\mathbb{E}[X_{i}])
\leq 0, \quad\quad q.s.$$
that is
$$\mathbb{V}\left( \limsup_{n\rightarrow\infty} \frac{1}{A_{n}}
\sum_{i=1}^{n}a_{i}(X_{i}-\mathbb{E}[X_{i}])
> 0\right)=0.$$
Therefore, the proof of theorem is completed.
\end{proof}

The  Kolmogorov type and Marcinkiewicz-Zygmund type strong laws of large numbers for negatively associated random variables can be derived from the theorem above directly.

\begin{thm}(Kolmogorov SLLN)
Let $\{X_{i}\}_{i=1}^{\infty}$ be a sequence of negatively associated random variables on the upper probability space $(\Omega,\mathcal{F},\mathcal{P},\mathbb{V})$ such that
$\sup_{i\geq 1}\mathbb{E}[|X_{i}|^{\alpha+1}]<\infty$ for some constant $\alpha>0$. Then
$$\mathbb{V}\left(\left\{\liminf_{n\rightarrow \infty}
\frac{\sum_{i=1}^{n}(X_{i}-\mathcal{E}[X_{i}])}{n}<0\right\}
\bigcup\left\{\limsup_{n\rightarrow \infty}
\frac{\sum_{i=1}^{n}(X_{i}-\mathbb{E}[X_{i}])}{n}>0\right\}\right)=0,
$$
also
$$v\left(\liminf_{n\rightarrow \infty}
\frac{\sum_{i=1}^{n}(X_{i}-\mathcal{E}[X_{i}])}{n}\geq0\right)=1,
\quad\quad
v\left(\limsup_{n\rightarrow \infty}
\frac{\sum_{i=1}^{n}(X_{i}-\mathbb{E}[X_{i}])}{n}\leq0\right)=1.
$$
\end{thm}

\newpage

\begin{thm}(Marcinkiewicz-Zygmund SLLN)
Let $\{X_{i}\}_{i=1}^{\infty}$ be a sequence of negatively associated random variables on the upper probability space $(\Omega,\mathcal{F},\mathcal{P},\mathbb{V})$ such that
$\sup_{i\geq 1}\mathbb{E}[|X_{i}|^{\alpha+1}]<\infty$ for some constant $\alpha>0$.
Then for $1\leq p<1+\alpha$,
$$\mathbb{V}\left(\left\{\liminf_{n\rightarrow \infty}
\frac{\sum_{i=1}^{n}(X_{i}-\mathcal{E}[X_{i}])}{n^{1/p}}<0\right\}
\bigcup\left\{\limsup_{n\rightarrow \infty}
\frac{\sum_{i=1}^{n}(X_{i}-\mathbb{E}[X_{i}])}{n^{1/p}}>0\right\}\right)=0,
$$
also
$$v\left(\liminf_{n\rightarrow \infty}
\frac{\sum_{i=1}^{n}(X_{i}-\mathcal{E}[X_{i}])}{n^{1/p}}\geq0\right)=1,
\quad\quad
v\left(\limsup_{n\rightarrow \infty}
\frac{\sum_{i=1}^{n}(X_{i}-\mathbb{E}[X_{i}])}{n^{1/p}}\leq0\right)=1.
$$
\end{thm}

\section{Strassen type invariance principle}

In this section, we derive  Strassen type invariance principles  of   strong limit theorems of large numbers for negatively associated and vertically independent random variables  as applications of our main results.

\begin{thm}
Let $\{X_{i}\}_{i=1}^{\infty}$ be a sequence of negatively associated random variables on the upper probability space $(\Omega,\mathcal{F},\mathcal{P},\mathbb{V})$ such that
$\sup_{i\geq 1}\mathbb{E}[|X_{i}|^{\alpha+1}]<\infty$ for some constant $\alpha>0$. Let $\{a_{i}\}_{i=1}^{\infty}$ be a sequence of bounded positive numbers  and $\{A_{n}\}_{n=1}^{\infty}$  be a sequence of positive numbers such that
$$\lim_{n\rightarrow \infty}\frac{A_{n}}{n^{\frac{1}{\beta+1}}}=\infty,$$
where $\beta$ is a constant such that $\beta\in\left(0,\min(1,\alpha)\right)$.
Then for any continuous function $\varphi(\cdot)$ on $\mathbb{R}$,
$$v\left(\limsup_{n\rightarrow\infty}\varphi\left(
\frac{\sum_{i=1}^{n}a_{i}(X_{i}-\mathbb{E}[X_{i}])}{A_{n}}
\right) \leq \sup_{x\leq 0}\varphi(x)  \right)=1,$$
$$v\left(\liminf_{n\rightarrow\infty}\varphi\left(
\frac{\sum_{i=1}^{n}a_{i}(X_{i}-\mathcal{E}[X_{i}])}{A_{n}}
\right) \geq \inf_{x\geq 0}\varphi(x)  \right)=1.$$
\end{thm}

\begin{proof}
Set
$$A:=\left\{\omega\bigg| \limsup_{n\rightarrow\infty}
\frac{\sum_{i=1}^{n}a_{i}(X_{i}(\omega)-\mathbb{E}[X_{i}])}{A_{n}}
 \leq 0                \right\}.$$
Fixed $\omega\in A$, then for any $\epsilon>0$, there exists a constant $N_{\omega}(\epsilon)\in\mathbb{N}^{+}$ such that for any $n>N_{\omega}(\epsilon)$, we have
$$\sup_{m\geq n }\frac{\sum_{i=1}^{m}a_{i}(X_{i}(\omega)-\mathbb{E}[X_{i}])}{A_{m}}
<\epsilon. $$
Therefore,
\begin{equation*}
\begin{split}
&\quad \ \limsup_{n\rightarrow\infty}\varphi\left(
\frac{\sum_{i=1}^{n}a_{i}(X_{i}(\omega)-\mathbb{E}[X_{i}])}{A_{n}}
\right)\\
&=\lim_{n\rightarrow\infty}\left(\sup_{m\geq n}
\varphi\left(
\frac{\sum_{i=1}^{m}a_{i}(X_{i}(\omega)-\mathbb{E}[X_{i}])}{A_{m}}\right)\right)\\
&\leq\lim_{n\rightarrow\infty}\left(\sup_{x<\epsilon}\varphi(x)\right)\\
&=\sup_{x<\epsilon}\varphi(x).
\end{split}
\end{equation*}
Then by the arbitrariness of $\epsilon$ and continuity of $\varphi$, for fixed $\omega$, we achieve
$$\limsup_{n\rightarrow\infty}\varphi\left(
\frac{\sum_{i=1}^{n}a_{i}(X_{i}(\omega)-\mathbb{E}[X_{i}])}{A_{n}}
\right)\leq \sup_{x\leq 0}\varphi(x).$$
Furthermore, by the arbitrariness of $\omega\in A$, we obtain
$$A\subseteq\left\{ \omega\bigg|  \limsup_{n\rightarrow\infty}\varphi\left(
\frac{\sum_{i=1}^{n}a_{i}(X_{i}-\mathbb{E}[X_{i}])}{A_{n}}
\right)\leq \sup_{x\leq 0}\varphi(x)\right\}.$$
 Applying Theorem \ref{mainthm}, we conclude that $v(A)=1$. Therefore,
$$v\left(  \limsup_{n\rightarrow\infty}\varphi\left(
\frac{\sum_{i=1}^{n}a_{i}(X_{i}-\mathbb{E}[X_{i}])}{A_{n}}
\right)\leq \sup_{x\leq 0}\varphi(x)  \right)=1.$$
Similarly, we can achieve that
$$v\left(\liminf_{n\rightarrow\infty}\varphi\left(
\frac{\sum_{i=1}^{n}a_{i}(X_{i}-\mathcal{E}[X_{i}])}{A_{n}}
\right) \geq \inf_{x\geq 0}\varphi(x)  \right)=1.$$
\end{proof}

\begin{remark}
The Strassen type  invariance principle also holds for vertically independent random variables on  the upper probability
space $(\Omega,\mathcal{F},\mathcal{P},\mathbb{V})$.
\end{remark}

\section*{References}

\bibliography{mybibfile}

\end{document}